\documentclass[12pt,a4]{amsart}
\usepackage{amsmath,amssymb,amsfonts, amscd,enumerate}
\usepackage{amsbsy}
\usepackage{amsthm}
\usepackage{amstext}
\usepackage{amsopn}
\usepackage[german, english]{babel}
\usepackage{hyperref}
\usepackage{mathrsfs}
\usepackage{graphicx}
\usepackage{latexsym}
\usepackage{color}

\usepackage[utf8]{inputenc}

\newcommand{\Hmm}[1]{\leavevmode{\marginpar{\tiny%
$\hbox to 0mm{\hspace*{-0.5mm}$\leftarrow$\hss}%
\vcenter{\vrule depth 0.1mm height 0.1mm width \the\marginparwidth}%
\hbox to 0mm{\hss$\rightarrow$\hspace*{-0.5mm}}$\\\relax\raggedright #1}}}

  \def\Efun{\mathscr{E}}

  \def\Efun{\mathcal{E}}

\newtheorem{thm}{Theorem}[section]

\newtheorem{lemma}[thm]{Lemma}

\theoremstyle{definition}

\newtheorem{eg}[thm]{Example}
\newtheorem{rem}[thm]{Remark}

\numberwithin{equation}{section}
\numberwithin{thm}{section}

\newcommand{\R}{{\mathbb R}}

\newcommand{\N}{{\mathbb N}}

\newcommand{\al}{{\alpha}}

\newcommand{\de}{{\delta}}

\newcommand{\gm}{{\gamma}}

\newcommand{\si}{{\sigma}}
\newcommand{\lm}{{\lambda}}
\newcommand{\ph}{{\varphi}}
\newcommand{\Deg}{{\mathrm{Deg}}}

\newcommand{\lmin}{{\lambda}^{(0)}_{p}}
\newcommand{\lmax}{{\lambda}^{(1)}_{p}}

\newcommand{\lmaxq}{{\lambda}^{(1)}}

\newcommand{\hmin}{{h}^{(0)}_{p}}
\newcommand{\hmax}{{h}^{(1)}_{p}}

\newcommand{\hminq}{{h}^{(0)}}
\newcommand{\hmaxq}{{h}^{(1)}}

\begin{document}
\title[Cheeger inequalities for  $p$-Laplacians]
{General Cheeger inequalities for  $p$-Laplacians on   graphs}
\author[M.~Keller]{Matthias Keller}\address{M. Keller, Mathematisches Institut \\Friedrich Schiller Universit{\"a}t Jena \\07743 Jena, Germany } \email{m.keller@uni-jena.de}
\author[D.~Mugnolo]{Delio Mugnolo}\address{Delio Mugnolo, Lehrgebiet Analysis, FernUniversit\"at in Hagen, 58084 Hagen, Germany}
\email{delio.mugnolo@fernuni-hagen.de}

\thanks{
This  research has been partially supported by the Land
Baden-W\"urttemberg in the framework of the
\emph{Juniorprofessorenprogramm} -- research project on ``Symmetry
methods in quantum graphs'' and the DFG in the framework of the
project ``Geometry of discrete spaces and spectral theory of
non-local operators''.}

\maketitle

\begin{abstract}
We prove Cheeger inequalities for $p$-Laplacians on finite and
infinite weighted graphs. Unlike in previous works, we do not impose
boundedness of the vertex  degree, nor do we restrict ourselves to
the normalized Laplacian and, more generally, we do not impose any
boundedness assumption on the geometry. This is achieved by a novel
definition of the measure of the boundary which is using the idea of
intrinsic metrics.
 For the non-normalized case, our bounds on the spectral gap of  $p$-Laplacians
 are already significantly better  for finite graphs and for
infinite graphs they yield non-trivial bounds even in the case of
unbounded vertex degree. We, furthermore, give upper bounds  by the
Cheeger constant and by the exponential volume growth of distance
balls.
\end{abstract}

\section{Introduction}
Cheeger inequalities have a long history and are relevant for both pure mathematics and applied mathematics. The pure mathematical interest stems from the fact that they connect geometry and spectral theory. In applications they are  used to partition the underlying space in an efficient way.

From the perspective of pure mathematics  the history of our topic
starts with the work of Cheeger \cite{Che70}. On compact manifolds
Cheeger used an isoperimetric constant to estimate the first
non-trivial eigenvalue of the Laplacian from below. This
isoperimetric constant -- thus called \emph{Cheeger constant} ever
since -- serves as measure for separating  the manifold  into two
approximately equally sized  parts.

Similar ideas for finite graphs were independently found shortly
afterwards in the pioneering work of Fiedler~\cite{Fie73},  where
the first non-trivial eigenvalue of the graph Laplacian is shown to
be a quantitative measure of the graph's connectedness. The first
``genuine'' Cheeger estimates on graphs are due to Dodziuk
~\cite{Dod84} and Alon/Milmann~\cite{AloMil85}. Since then these
estimates have been improved and various variants have been shown.
However, it was only until recently that non-trivial estimates for
unbounded graph Laplacians were available. Specifically, in previous
investigations either it was the \emph{normalized} graph Laplacian
(which is always a bounded operator) that was considered, or else an
upper bound on the vertex degree appeared in the denominator of the
lower bound, thus making the inequality trivial whenever the degree
is unbounded. In \cite{BauKelWoj15} a novel measure of the boundary
of a set has been introduced using the concepts of intrinsic metrics
for non-local Dirichlet forms. These metrics have first been
systematically studied by Frank/Lenz/Wingert \cite{FraLenWin14}
for general regular Dirichlet forms. Since then they have proven a
very efficient tool, see e.g.\ the recent survey~\cite{Kel15} on
graphs.

The history sketched above for the classical case of the linear
Laplacian has inspired analogs in non-linear theory. After Cheeger
\cite{Che70} treated the linear case $p=2$ and Yau~\cite{Yau75}
proved an equality for $p=1$, Kawohl/Fridman~\cite{KawFri03}
generalized Cheeger's inequality to the $p$-Laplace-Beltrami
operators, for $p>1$. Cheeger inequalities for the $p$-Laplacian (or
the normalized $p$-Laplacian) on finite graphs can be found in
\cite{Amg03,Tak03,BuhHei09}.

The applications perspective is converse. Here, one is interested in
finding graph partitions. While  computing the Cheeger constant of a
graph is an NP-hard problem, see e.g.~\cite{Kai04}, the computation
of the first non-trivial eigenvalue and the corresponding
eigenfunction is rather efficient{, by simple variational methods}.
Thus,  Fiedler's intuition \cite{Fie73,Fie75} had far-reaching
repercussions in theoretical computer science. {In particular, the
supports of} the positive and negative part of the first non-trivial
eigenfunction {of the graph Laplacian (or $p$-Laplacian)} suggest a
reasonable splitting of the graph. In view of the Cheeger
inequalities this splitting is close to the optimal Cheeger cut.
Indeed, several machine learning tasks -- like clustering of data
sets or segmenting of pictures -- can actually be reduced to the
study of eigenvalues of the Laplacian on the associated graphs.
Since the pioneering investigations in~\cite{DonHof73,Fie73},
spectral methods on graphs  or manifolds associated with data sets
have become rather popular in computer science, cf.
e.g.~\cite{NgJorWei01,Lux07,GraPol10,Bol13}.

The $p$-Laplacians have recently aroused interest in applications to
computer science mostly because their lowest nontrivial eigenvalue
converges to the Cheeger constant as $p\to1$. This was shown in the
continuous case in the remarkable work of
{Kawohl/Fridman~\cite{KawFri03}, which was later proven by
B\"uhler/Hein \cite{BuhHei09} in the graph setting}. Hence, the
Cheeger constant can  be approximated by means of a sequence of
convex optimization problems.

As in the linear case, the known estimates for the $p$-Laplacian on
graphs are proven either for the normalized Laplacian, or else they
involve an upper bound on the vertex degree in the estimate. In the
second case, this leads to non-optimal estimates for finite graphs
that have only few vertices of very large degree, like real-life
scale-free networks. In the case of infinite graphs with unbounded
degree the estimates known so far even become trivial.

In this paper, we adapt the ideas of intrinsic metrics  from
\cite{BauKelWoj15} to the non-linear case of  $p$-Laplacians to
improve the estimates known so far. The techniques use a novel
definition of the boundary measure of a set. In particular, not only
the weight of an edge is taken into account but also its length.
This length stems from a metric whose $p^{*}$-norm, with
$1/p+1/p^{*}=1$, of the ``discrete gradient'' is less than one. In
the linear case it can be motivated by the distances attributed to a
diffusion on the graph, see \cite{Kel15}.  One instance  of such a
metric can be obtained involving the inverses of the vertex degrees,
see  Example~\ref{example}. From this perspective the vertex degrees
are part of the minimization itself and do not enter as a uniform
upper bound.

Our main perspective is  rather the one of pure mathematics, that is
we look for estimates of spectral quantities in terms on geometric
ones. Nevertheless, the  Cheeger constant defined by these novel
metrics might be of applicative interest on its own right, as it
encodes relevant geometric data of the underlying graph.

We also prove upper bounds for the first non-trivial eigenvalue.
Such estimates  are known as Buser inequalities in the case of
manifolds. Unlike in  the manifold case there is classically no
curvature notion whatsoever entering the upper estimate in the
graph case~\cite{AloMil85,Moh88}. However, we get an upper bound
involving a constant related to uniform discreteness of the space.
Indeed, it turns out that this estimate becomes often trivial when
the vertex degree is unbounded. This once again suggests that a
lower bound on the curvature in the case of manifolds corresponds to
an upper bound on the vertex degree for graphs.

We, furthermore, give an alternative proof of B\"uhler/Hein's
approximation result for finite graphs. Finally, we show an upper
bound for bottom of the spectrum in terms of exponential volume
growth of balls. Classically this is known as Brook's theorem
\cite{Bro81} and was shown for regular Dirichlet forms in
\cite{HaeKelWoj13} in the linear case (i.e., $p=2$). Our proof uses
again  an adaption of the concept of intrinsic metrics to the case
of general $p$.

The paper is structured as follows. In the next section we introduce
the set-up with all relevant quantities. This is followed by
Section~\ref{s:Cheeger} where we formulate and prove the Cheeger
inequalities. In Subsections~\ref{s:upper} and~\ref{s:convergence}
of Section~\ref{s:Cheeger} we prove upper bounds in the sense of a
Buser inequality and the convergence result in the case of finite
graphs. In Section~\ref{s:brooks} we prove the upper bounds by
exponential volume growth in the spirit of a Brooks-type theorem. Finally, in the appendix we discuss the interpretation of our variational results as estimates on the spectral gap of discrete $p$-Laplacians.

\section{Set up}
\subsection{Graphs and the energy functional}
Let $X$ be a discrete  countably finite or countably infinite set.
We denote the set of real-valued functions on $X$ by $C(X)$ and its
subset of finitely supported functions  by $C_{c}(X)$.

Let a symmetric function $b:X\times X\to[0,\infty)$ with zero
diagonal be given such that
\begin{align*}
    \sum_{y\in X}b(x,y)<\infty,\qquad x\in X.
\end{align*}
Such a function has an interpretation in classical graph theory: the
elements of $X$ are vertices and two vertices $x,y\in X$ are
connected by an edge with weight $b(x,y)$ if and only $b(x,y)>0$; we
write $x\sim y$ in this case.

Let $m:X\to(0,\infty)$ be a function which extends to a measure via
additivity.  Moreover, we denote the spaces of $p$-summable real
valued functions on $X$ with respect to the measure $m$ by
$\ell^{p}(X,m)$ and the corresponding norm by $\|\cdot\|_{m,p}$,
$p\in[1,\infty)$. The dual pairing for functions $f\in\ell^{p}(X,m)$
and $g\in \ell^{q}(X,m)$ with $p,q\in(1,\infty)$ and $1/p+1/q=1$ is
given by
\begin{align*}
    \langle{f},{g}\rangle_{m}:=\sum_{X}fgm=\sum_{x\in
    X}f(x)g(x)m(x).
\end{align*}

Following Nakamura/Yamasaki  \cite{NakYam76}, for $p\in[1,\infty)$,
we introduce a convex energy functional
$\mathcal{E}_{p}:C(X)\to[0,\infty]$ by
\begin{align*}
    \Efun_p(f)
    :=\frac{1}{2}\sum_{x,y\in X}b(x,y)|f(x)-f(y)|^{p}
    ,\qquad f\in C(X)\ .
\end{align*}

In this paper we are interested in giving estimates on the quantity
\begin{align*}
    \lmin:=
    \inf_{0\neq\ph\in C_{c}(X)}\frac{\Efun_{p}(\ph)}{\|\ph\|_{m,p}^{p}}.
\end{align*}
and in the case $m(X)<\infty$
\begin{align*}
\lmax&:=\inf_{\mathrm{const}\neq f\in
\ell^{p}(X,m)}\frac{\Efun_{p}(f)}{\inf_{\gm\in\R}\|f-\gm\|_{m,p}^{p}}.
\end{align*}
 The latter is of
course in particular relevant in the case of finite graphs.

\subsection{A non-linear generalization of intrinsic metrics}
In recent years the notion of intrinsic metrics has been developed
for graphs; it has various strong applications for the case $p=2$.
Here we are going to extend it to general $p$. For our purposes we
do not need to enforce the triangle inequality: This suggests {to
introduce} the set
\begin{align*}
R_{p}(b,m):=\{&d:X\times X\to[0,\infty)
    \mid\mbox{symmetric and }\\
    &\sum_{y\in X}b(x,y)d(x,y)^{p/(p-1)}\leq m(x) \mbox{, for all } x\in X\}
\end{align*}
{for $p\in (1,\infty)$ and
\begin{align*}
R_{1}(b,m):=\{&d:X\times X\to[0,\infty)  \mid\mbox{symmetric and
}d(x,y)\leq 1 \mbox{ if }x\sim y\}
\end{align*}
for $p=1$.}

Let  us give two examples: on one hand we show that $R_{p}(b,m)$
does indeed contain non-vanishing functions for any graph structure
$b$ and any measure $m$; on the other hand we show how to embed a
classical object of the literature -- the combinatorial graph
distance -- in our theoretical setting.

\begin{eg}\label{example}
(a) For a given graph $b$ over $(X,m)$ and $p\in(1,\infty)$ we
define a function $d_{p}:X\times X\to[0,\infty)$ by
\begin{align*}
    d_{p}(x,y):=(\mathrm{Deg}(x)\vee\mathrm{Deg}(y))^{-(p-1)/p},\quad
    x,y\in X,
\end{align*}
where $\mathrm{Deg}(x)$ is the weighted degree of the vertex $x\in
X$ with respect to $b$ and $m$ given by
\begin{align*}
    \mathrm{Deg}(x)=\frac{1}{m(x)}\sum_{y\in X}b(x,y),\quad x\in X.
\end{align*}
The function $d_p$ can be seen to be in  $R_{p}(b,m)$ for any $b $
and $m$ by direct calculations. Furthermore, we can easily construct
a pseudo metric from $d_{p}$ by considering the associated path
metric.

(b) Let a graph $b$ over $X$ be given. In the case where $m$ is
chosen to be the \emph{normalizing measure}, i.e.,
\begin{align*}
    m(x)=\sum_{y\in X}b(x,y),
\end{align*}
{and, hence, $\mathrm{Deg}\equiv 1$, then} the combinatorial graph
distance $d_{\mathrm{comb}}$ defines a function in $R_{p}(b,m)$ for
all $p\ge 1$. This is case most usually considered in the
literature, see e.g. \cite{Amg03, Tak03}.
\end{eg}


\begin{rem}\label{rem:tak}
For $p=2$, a function in $R_{p}(b,m)$ that is additionally a pseudo
metric, i.e., satisfies the triangle inequality, is called an
\emph{intrinsic metric}. This concept was first studied
systematically for general regular Dirichlet forms by
Frank/Lenz/Wingert in  \cite{FraLenWin14} and used since then in
various contexts, see
e.g.~\cite{BauKelWoj15,Fol11,GriHuaMas12,HuaKelMas13}. Here, we
replace $2$ by $q=p/(p-1)$ which is the conjugate of $p$, i.e.,
$1/p+1/q=1$.
\end{rem}

\section{Cheeger inequalities}\label{s:Cheeger}
We start by defining the isoperimetric constants with the novel
definition of  the boundary measure of a set. This is inspired by
\cite{BauKelWoj15} where this was used in the case $p=2$.

Below, we introduce the isoperimetric constants $\hmin$ and $\hmax$
which are each tailored for the quantities $\lmin$  and $\lmax$.  We
then state our main results: Theorem~\ref{t:main1} for $\lmin$ and
Theorem~\ref{t:main2} for   $\lmax$. The latter result is in
particular relevant for finite graphs.

For the proof we show an abstract Cheeger estimate,
Theorem~\ref{t:isoperimetric_estimate}, from which we derive both
Theorem~\ref{t:main1} and Theorem~\ref{t:main2}.

Afterwards, we proceed by showing upper bounds for  the quantities
$\lmin$ and $\lmax$ in terms of $\hmin$ and $\hmax$ and a uniform
discreteness constant. Finally, we turn to finite graphs and give an
alternative proof of the convergence of $\lmax$ to $\lmaxq_{1}$ as
$p\to1$.

\subsection{Isoperimetric constant and results}
\label{s:isoperimetric_constant}
In this subsection we define the isoperimetric constants and state
the results.

Let $W\subseteq X$. We define the \emph{boundary} $\partial W$ of
the set $W$ by
\begin{align*}
\partial W:= W\times( X\setminus W).
\end{align*}
 The \emph{measure of the boundary} with respect to a function
$w:X\times X\to[0,\infty)$ is defined as
$${|\partial W|_{w}}:=w(\partial W)=\sum_{(x,y)\in \partial W}w(x,y).$$

Whenever a graph structure on $X$ given by some $b$ along with a
measure $m$ is considered, this definition will be used with $w=bd$,
where $d$ is a function in $R_{p}(b,m)$. In this case the sum over
$\partial W$ above is effectively only over the edges leaving $W$.
This definition of the measure of the boundary generalizes the
classical theory which considers $w=b$, i.e., $d=1$, only. This
generalization is the key idea so that we do not have to impose any
boundedness assumptions, neither by assuming bounded weighted vertex
degree nor by restricting ourselves to the case of the normalizing
measure (cf. Example~\ref{example}).

 We define for $p\in[1,\infty)$ the
$p$-\emph{isoperimetric numbers} $\hmin$ and $\hmax$ by
\begin{align*}
\hmin:=\sup_{d\in R_{p}(b,m)}h^{(0)}(d)\quad\mbox{with}\quad
h^{(0)}(d):=\inf_{W\subseteq X \mbox{{ \scriptsize
finite}}}\frac{{|\partial W|_{bd}}}{m(W)}.
\end{align*}
and, only in the case $m(X)<\infty$,
\begin{align*}
\hmax&:=\sup_{d\in R_{p}(b,m)}h^{(1)}(d)\quad\mbox{with}\quad
h^{(1)}(d):= \inf_{\substack{W\subseteq X \\
m(W)\leq {m(X)}/{2}}}\frac{{|\partial W|_{bd}}}{m(W)}.
\end{align*}
In the case of finite graphs one always has $\hmin=0$ since one can
choose $W=X$ in the definition of $h^{(0)}(d)$. Furthermore,
choosing $m$ as the normalizing measure, see
Example~\ref{example}.(b), the constants $\hmin$ and respectively
$\hmax$ agree with the classical Cheeger constants in the case of
infinite and respectively finite graphs.

Having introduced the relevant quantities we are in the position to
state {our main results}. These are two Cheeger-type inequalities
relating the isoperimetric numbers and the spectral gaps.

\begin{thm}\label{t:main1}
For all $p\in(1,\infty)$,
\begin{align*}
\frac{2^{p-1}}{p^{p}}{\hmin}^{p}\leq     \lmin.
\end{align*}
\end{thm}

For $p\in [1,\infty)$, let
\begin{align*}
    D_{p}
    =\{f\in C(X)\mid \mathcal{E}_{p}(f)<\infty\}\cap
    \ell^{p}(X,m).
\end{align*}

We say a function $f\in D_{p}$ is a \emph{weak solution} for $\lm$
if for all $g\in D_{p}$
\begin{equation}\label{eq:weaksol}
\begin{split}
\mathcal{E}_{p}'(f)g&:=\frac{1}{2}\sum_{x,y\in
X}b(x,y)|f(x)-f(y)|^{p-2}(f(x)-f(y))(g(x)-g(y))\\
&=\lm\sum_{ X}{|f|^{p-2}}f {g}{m}.
\end{split}
\end{equation}

\begin{thm}\label{t:main2}
Let $m(X)<\infty$. Let there be a non-constant weak solution $f\in
D_{p}$ for $\lmax$  for some $p\in(1,\infty)$. Then,
\begin{align*}
     \frac{2^{p-1}}{p^{p}} {\hmax}^{p}\leq \lmax.
\end{align*}
\end{thm}

\begin{rem}\label{rem:BH}
(a) Using the Fr\'{e}chet differentiability of $\mathcal{E}_p $ on suitable Banach spaces, see e.g.  \cite{Mug13}, one can prove the following:
In the case  weak solutions for $\lmax $ exist, they are exactly the minimizers of the functional $\mathcal E_p$.
This was already  shown in  in~\cite[Thm.~3.1]{BuhHei09} for finite graphs. Indeed, $ \lmin$ and respectively $ \lmax$ can be understood as the bottom of the spectrum and respectively first non-trivial eigenvalue of  $p$-Laplacian under Dirichlet and respectively Neumann conditions. These $p$-Laplacians are restrictions of the general $p$-Laplacian that is introduced in the appendix.

(b) Theorem~\ref{t:main2}  can be seen as a generalization of the
corresponding estimates in~\cite{Amg03,Tak03} and \cite{BuhHei09} in
the case of finite graphs. For $b$ over $(X,m)$ let the classical
Cheeger constant be given by
\begin{align*}
 h:=\inf_{m(W)\leq m(X/2)}\frac{|\partial W|_{b}}{m(W)}.
\end{align*}
In~\cite{Amg03,Tak03} the case
\begin{itemize}
\item $b:X\times X\to\{0,1\}$,
\item $m(x)= \sum_{y\in X}b(x,y)=\#\{y\in X\mid x\sim y\}$, $x\in
X$,
\end{itemize}
was considered and the bound
\begin{align*}
2^{p-1}\left({\frac{h}{p}}\right)^{p}\leq \lmax
\end{align*}
 was obtained.
This is a special case of Theorem~\ref{t:main2}  since
$h=h^{(1)}(d_{\mathrm{comb}})$ with $d_{\mathrm{comb}}$ being the
combinatorial graph metric   which is in $R_{p}(b,m)$ for $b$ and
$m$ as chosen above.

Furthermore, for
\begin{itemize}
\item $b:X\times X\to\{0,1\}$,
\item $m\equiv 1$
\end{itemize}
we improve the bound in  \cite[Theorem~4.3]{BuhHei09}, where the
inequality
\begin{align*}
\left({ \frac{2}{M}}\right)^{p-1}\left({\frac{h}{p}}\right)^{p}\leq
\lmax
\end{align*}
was proven with $M:=\sup_{x\in X}\#\{y\in X\mid x\sim y\}$. Observe
that the combinatorial graph metric is not in $R_{p}(b,1)$ apart
from the trivial case of a graph consisting of isolated vertices and
edges. To see that our estimate is  sharper, one can choose the
weight $d_p(x,y):=(\Deg(x)\vee\Deg(y))^{-(p-1)/p}$ from
Example~\ref{example}. In Example~\ref{example2} below we give
explicit constructions of graphs where our estimate is seen to be
significantly sharper than the one of  \cite{BuhHei09}.
\end{rem}

\begin{eg}\label{example2} We consider
a finite $k$-regular graph, i.e.,  $b_{0}:X\times X\to\{0,1\}$ such
that $\sum_{y\in X}b(x,y)=\#\{y\sim x\}=k$ for all $x\in X$.
Furthermore, let $m\equiv 1$. Let $W_{0}$ be a set that minimizes
\[
h_{0}:=\inf_{\#W\leq \# X/2}\frac{|\partial W|_{b_{0}}}{\#W}
\]
and set $N_{0}:=\#W_{0}-1$. We may assume that the graph is such
that $N_{0}\ge k$.

Now, we let $b$ be the graph over $X$ which  has the edges of
$b_{0}$ and we choose an arbitrary vertex $w\in W_{0}$ and connect
it to every other vertex in $W_0$ by an edge.

Obviously, the Cheeger constant $h:=\inf_{\#W\leq \#
X/2}\frac{|\partial W|_{b}}{\#W}$ of the new graph $b$ equals
$h_{0}$ and $W_{0}$ is still a minimizing set since $|\partial
W|_{b}\ge |\partial W|_{b_{0}}$  for any $W\subseteq X$  and
$|\partial W_{0}|_{b}= |\partial W_{0}|_{b_{0}}$. In other words
 the \emph{conductivity} $\frac{|\partial
W|_{b}}{m(W)}$ of a set $W$ depends on the connectivity to its
complement $X\setminus W$ as well as to its measure, but not on the
internal structure of $W$. So, \cite{BuhHei09} yields the estimate
\begin{align*}
\left({\frac{2}{N_{0}}}\right)^{p-1}\left({\frac{h}{p}}\right)^{p}\leq
\lmax.
\end{align*}

To compare this to our estimate, we choose the function
$d_{0}:=bk^{-1/q}$ with $q=p/(p-1)$. It is easy to see that
$d_{0}\in R_{p}(b,m)$ and $|\partial W|_{bd_{0}}=k^{-1/p^*}|\partial
W|_{b_{0}}$. Hence, $\hmaxq(d_{0})=h_{0}=h$ and the estimate
\begin{align*}
\frac{2^{p-1}}{k^{1/q}}\left({\frac{h }{p}}\right)^{p}\leq \lmax
\end{align*}
improves the bound of  \cite{BuhHei09} above by a factor
$k^{1/q}/N_{0}^{p-1}$.

Of course, the choice of $d_{0}$ above required a rather detailed
knowledge of the structure of the graph. Below we show that even
for the generic function $d$ from Example~\ref{example}.(a) the bound is
improved.

So, consider the function from Example~\ref{example}.(a)
\[
d(x,y)=(\mathrm{Deg}(x)\vee\mathrm{Deg}(y))^{-1/q}
\]
with $q=p/(p-1)$ and $\mathrm{Deg}$ given by
$\mathrm{Deg}(x)=\sum_{y\in X}b(x,y)$ in the case $m\equiv 1$. Then,
for any $W\subseteq X$ with $w\in W$ we have
\begin{align*}
|\partial W|_{bd}&=\sum_{(x,y)\in W\times X\setminus W, x\neq w}
bd(x,y)+\sum_{y\in  X\setminus W}bd(w,y)\\
&=(k+1)^{-1/q}\hspace{-1cm}\sum_{(x,y)\in W\times X\setminus W,
x\neq w}
b(x,y)+(k+N_{0})^{-1/q}\sum_{y\in  X\setminus W}b(w,y)\\
&\ge (k+1)^{-1/q}(|\partial W|_{b_{0}}-k+1)
\end{align*}
for any $W\subseteq X$ with $w\in X\setminus W$, we have analogously
\begin{align*}
|\partial W|_{bd} &\ge (k+1)^{-1/q}(|\partial W|_{b_{0}}-k+1).
\end{align*}
Hence, with $c=(k+1)^{-1/q}(k-1)$ we obtain
\begin{align*}
    \hmax(d)=\inf_{\#W\leq \# X/2}\frac{|\partial W|_{bd}}{\#W}
    \geq {(k+1)^{-1/q}}
    \inf_{\#W\leq \# X/2}\frac{|\partial
    W|_{b}-c}{\#W}.
\end{align*}
So, whenever we have a graph where $N_{0}=\#W_{0}-1$  is
significantly larger than $(k+1)^{1/q}h_{0}^{-{1}}$, then
$\hmaxq{}(d)$ is  close to $(k+1)^{-1/q}h_{0}$. In such cases our
estimate is still significantly better than the one of
\cite{BuhHei09} above, namely by the factor $k^{1/q}/N_{0}^{p-1}$.
\end{eg}

\subsection{A general isoperimetric inequality}
In this subsection we prove a general isoperimetric inequality from
which we will deduce  Theorem~\ref{t:main1} and Theorem~\ref{t:main2}.

Let $p\in [1,\infty)$ be given. We extend as usual a symmetric
function $w:X\times X\to[0,\infty)$ to a measure   by
$$w(U):=\sum_{(x,y)\in U\times U}w(x,y),\qquad U\subseteq X\times X\ .$$
Furthermore, let $m:X\to(0,\infty)$ be given. Moreover, for a linear
subspace $G\subseteq C(X)$, we let
 $$Y_{G}=\{ \Omega_{t}{(|f|^{p})}\mid f\in G, t>0\}, $$
 where,   for a function $g\in C(X)$, the level sets
 $\Omega_{t}(g)$ are given by
\begin{align*}
    \Omega_{t}{(g)}:=\{x\in X\mid g(x)>t\}.
\end{align*}

Given the  ingredients $p$, $w$, $m$ and $G$, we define a general
isoperimetric constant
\begin{align*}
    h_{p,w,m,G}:=
    \inf_{\substack{W\in Y_G\\ m(W)<\infty}} \frac{w(\partial W)}{m(W)}.
\end{align*}

For example if $G=C_{c}(X)$,  then $Y_{C_{c}(X)}$ consists of all
finite sets. Furthermore, for $G=\ell^{p}(X,m)$, $p\in[1,\infty)$,
the set $Y_{\ell^{p}(X,m)}$ is  the set of all finite measure
subsets of $X$.

\begin{thm}\label{t:isoperimetric_estimate}
Let $p\in(1,\infty)$, $b$ be a graph over $(X,m)$, $w,\sigma:X\times
X\to[0,\infty)$ such that $w\leq b\sigma$, $G\subseteq C(X)$ and
\[
k(x):=\sum_{y\in X}(b\sigma^{p/(p-1)})(x,y),\qquad x\in X.
\]
Then, for all  $g\in G\cap\ell^{p}(X,m)$
\begin{align*}
{ \frac{ 2^{p-1}}{{p^{p}}}  { h_{p,w,m,G}^{p}} \|g\|_{m,p}^{p^{2}}}
\leq \|g\|_{k,p}^{p(p-1)} \Efun_{p}(g),
\end{align*}
where both sides may take the value $+\infty$.
\end{thm}

The inequality in Theorem~\ref{t:isoperimetric_estimate} bears some
resemblance to the interpolation inequality of Gagliardo/Nirenberg
on domains. Its proof is based on a co-area formula and the area
formula (or Cavalieri's principle). For a proof of the following two
lemmata see \cite[Theorem~12 and~13]{KelLen10}.

\begin{lemma}[Co-area formula]\label{l:coarea}
Let $w:X\times X\to[0,\infty)$ and $f:X\to[0,\infty)$. Then,
\begin{align*}
\frac{1}{2}\sum_{x,y\in X}w(x,y)|f(x)-f(y)| =
\int_{0}^{\infty}w(\partial\Omega_{t}{(f)})dt,
\end{align*}
where both sides may take the value $\infty$.
\end{lemma}

\begin{lemma}[Area formula]\label{l:area}
Let $m:X\to[0,\infty)$ and $f:X\to[0,\infty)$. Then,
\begin{align*}
    \sum_{x\in X}m(x)f(x)= \int_{0}^{\infty} m(\Omega_{t}{(f)})dt,
\end{align*}
where both sides may take the value $\infty$.
\end{lemma}

In contrast to the continuous setting there is no chain rule in the
discrete. The lemma below serves as a proxy of the chain rule. It is
due to S.~Amghibech,~\cite[Lemma~3]{Amg03}. For the sake of being
self-contained we give a proof which is  slightly different from
Amghibech's proof and owes to \cite{HolSoa97}.
\begin{lemma}\label{l:f^p}
Let $f:X\to[0,\infty)$ and $x,y\in X$. Then, for all
$p\in[1,\infty)$,
\begin{align*}
  \left|f^{p}(x)-f^{p}(y)\right|\leq
  p \left(\frac{f^{p}(x)+f^{p}(y)}{2}\right)^{(p-1)/p}|f(x)-f(y)|.
\end{align*}
\end{lemma}
\begin{proof} The statement is trivial for  $p=1$, so assume
$p\in(1,\infty)$.
We assume without loss of generality $f(y)\leq f(x)$ and denote
$a=f(y)$, $b=f(x)$. Furthermore, the only non-trivial case is
$0<a<b$  which we assume so forth.
As the function $ x\mapsto x^{p/(p-1)}$  is convex on $ [0,\infty)$,
we obtain by Jensen's inequality
\[
\left( \frac{1}{p} \frac{b^p-a^p}{b-a}\right)^{\frac{p}{p-1}}
=\left(  \frac{\int_a^b t^{p-1}\ dt}{b-a}\right)^{\frac{p}{p-1}} \le
\frac{\int_a^b t^p \ dt}{b-a} = \frac{1}{p+1}
\frac{b^{p+1}-a^{p+1}}{b-a}.
\]
We proceed by identity
\begin{align*}
    b^{p+1}-a^{p+1}=(b-a)\left(b^{p}+a^{p}\right) +ab\left(b^{p-1}-a^{p-1}\right)
\end{align*}
which leaves us to estimate the term $ab(b^{p-1}-a^{p-1})$. Note
that this term is non-negative for all $p\ge1$. Moreover, the
function $t\mapsto t^{-p}$ is convex on $(0,\infty)$ and, thus, its
image lies below the line segment connecting the points
$(b^{-1},b^{p-1})$ and $(a^{-1},a^{p-1})$. Therefore, for $p>1$, we
estimate
\begin{align*}
    \frac{1}{(p-1)}\left(b^{p-1}-a^{p-1}\right)
    &=\int_{b^{-1}}^{a^{-1}}t^{-p}dt\\
    &\leq \left(a^{-1}-b^{-1}\right)
    \left(\frac{\left(b^{p}-a^{p}\right)}{2}+a^{p}\right)\\
    &=\frac{1}{2ab}(b-a)\left(b^{p}+a^{p}\right).
\end{align*}
The inequality combined with the inequality and the equality above
yields the statement.
\end{proof}

\begin{proof}[Proof of Theorem~\ref{t:isoperimetric_estimate}]
We calculate using the co-area formula, Lemma~\ref{l:area},  and the
area formula, Lemma~\ref{l:coarea}, with $f=|g|^{p}$
\begin{align*}
   h_{p,w,m,G} \|g\|_{p}^{p}&=h_{p,w,m,G}\sum_{x\in X}m(x)|g(x)|^{p}\\
    &=h_{p,w,m,G}\int_{0}^{\infty}m(\Omega_{t}{(|g|^{p})})dt\\
    &\leq \int_{0}^{\infty}w(\partial \Omega_{t}{(|g|^{p})}) dt    \\
    &=\frac{1}{2}\sum_{x,y\in X}w(x,y)||g(x)|^{p}-|g(y)|^{p}|\ .
\end{align*}
Applying Lemma~\ref{l:f^p} we  conclude
\begin{align*}
\ldots&\leq  \frac{p}{2}\sum_{x,y\in X}w(x,y)
\left(\frac{|g(x)|^{p}+|g(y)|^{p}}{2}\right)^{(p-1)/p}||g(x)|-|g(y)||\\
&\leq\frac{p}{2}\sum_{x,y\in X} b(x,y)\sigma(x,y)
\left(\frac{|g(x)|^{p}+|g(y)|^{p}}{2}\right)^{(p-1)/p}|g(x)-g(y)|,
\end{align*}
where the last inequality follows from the assumption $w\leq
b\sigma$ and $||b|-|a||\leq|b-a|$. Applying H\"older's inequality
and using the definition $k(x)=\sum_{y\in
X}(b\sigma^{p/(p-1)})(x,y)$ yields
\begin{align*}
\ldots    \leq& \  \frac{p}{2}\left(\sum_{x\in X}
k(x)|g(x)|^{p}\right)^{({p-1})/{p}}\left(\sum_{x,y\in X}b(x,y)
|g(x)-g(y)|^{p}\right)^{{1}/{p}}\\
=&\  \frac{p}{2}\|g\|_{k,p}^{(p-1)}(2\Efun_{p}(g))^{1/p}.
\end{align*}
The statement follows now by  taking the $p$-th power and dividing
by $p^{p}/2^{p-1}$.
\end{proof}

\subsection{Proof of the main results}

\begin{proof} [Proof of Theorem~\ref{t:main1}]
Let $p\in(1,\infty)$, $q=p/(p-1)$ and  $ d\in R_{p}(b,m)$.  With the
notation of Theorem~\ref{t:isoperimetric_estimate}, let $\si= d$ and
$w=b\si$. Then, $k(x)=\sum_{y\in X}b(x,y) d^{q}(x,y)$, $x\in X$. As
$ d\in R_{p}(b,m)$, we estimate for $\ph\in C_{c}(X)$
\begin{align*}
    \|\ph\|^{p}_{k,p}=\sum_{x\in X}|\ph(x)|^{p}
    \sum_{y\in X}b(x,y) d^{q}(x,y)
    \leq \sum_{x\in X}|\ph(x)|^{p}m(x)
    =\|\ph\|_{m,p}^{p}
\end{align*}
By Theorem~\ref{t:isoperimetric_estimate} applied with $G=C_{c}(X)$,
we obtain for all $\ph\in C_{c}(X)$
\begin{align*}
\frac{2^{p-1}}{p^{p}} {    h_{p,bd,m,C_{c}(X)}^{p}}
\|\ph\|^{p^{2}}_{m,p} \leq \|\ph\|_{k,p}^{p(p-1)}\Efun_{p}(\ph)
\leq\|\ph\|_{m,p}^{p(p-1)}\Efun_{p}(\ph)
\end{align*}
and by definition $ h_{p,bd,m,C_{c}(X)}=\hminq{}(d)$. Hence,
\begin{align*}
   \frac{2^{p-1}}{p^{p}} {\hminq{}(d)}^{p}
   \leq\frac{\Efun_{p}(\ph)}{{\|\ph\|}_{m,p}^{p}}.
\end{align*}
By taking the supremum over all  $d\in R_{p}(b,m)$ and the infimum
over all $\ph\in C_{c}(X)$, we arrive at the statement
$\frac{2^{p-1}}{p^{p}} {\hmin}^{p} \leq\lmin.$
\end{proof}

\begin{proof} [Proof of Theorem~\ref{t:main2}]
Let $p\in(1,\infty)$. Let $f\in D_{p}$ be a non-constant weak
solution for $\lmax$. As $\mathcal{E}_{p}(f_{\pm})\leq
\mathcal{E}_{p}(f)$ and $\|f_{\pm}\|_{m,p}\leq\|f\|_{m,p}$ for all
$p\in[1,\infty)$, we find that the positive and negative parts
$f_{\pm}$ of $f$ are in $D_{p}$ as well. With the elementary
estimate
\begin{align*}
    |r_{+}-s_{+}|^{p}\leq |r-s|^{p-2}(r-s)(r_{+}-s_{+}),\quad r,s\in
    \R,
\end{align*}
we get for the positive part $f_{+}$
\begin{align*}
    \Efun_{p}(f_{+}) \leq \mathcal{E}_{p}'(f)f_{+}
    =\lmax\langle{f_{+},|f|^{p-2}f}\rangle_{m}
    =\lmax\|f_{+}\|_{m,p}^{p}
\end{align*}
Since we assumed that the solution is non-constant, we deduce
$\lmax\neq0$ and because $1\in D_{p}$ we infer by the definition of
weak solutions
\begin{align*}
    \sum_{X}|f|^{p-2}f m=    \sum_{X}|f|^{p-2}f 1m=\frac{1}{\lmax}\mathcal{E}'_{p}(f)1=0.
\end{align*}
Hence, $f$ has non-definite sign and we can assume without loss of
generality that the positive part $f_{+}=f\vee0$ of $f$ satisfies
$m(\mathrm{supp} f_{+})\leq m(X)/2$. For $G=\{g\in D_{p}\mid
m(\mathrm{supp} g)\leq m(X)/2 \}$ we have that $Y_{G}\subseteq
\{A\subseteq X\mid m(A)\leq m(X)/2\}$. Hence, $f_{+}\in G$ and
$h_{p,bd,m,G}\ge\hmaxq{}(d)$ for any $ d\in R_{p}(b,m)$.   By
Theorem~\ref{t:isoperimetric_estimate} applied $\si= d$ and
$w=b\si$,  and the considerations above we get
\begin{align*}
 \frac{2^{p-1}}{p^{p}} {\hmaxq{}(d)}^{p}
\|f_{+}\|^{p^{2}}_{m,p} \leq
\|f_{+}\|_{k,p}^{p(p-1)}\Efun_{p}(f_{+})
\leq\lmax\|f_{+}\|_{m,p}^{p^{2}},
\end{align*}
where we  used $\|f_{+}\|_{k,p}\le\|f_{+}\|_{m,p}$ in the last
estimate which is implied by $d\in R_p(b,m)$. Since $f$ has
non-definite sign as discussed above, we have $f_{+}\neq 0$. Thus,
dividing by $\|f_{+}\|_{m,p}^{p^{2}}$ and taking the supremum over
all $ d\in R_p(b,m)$ yields the result.
\end{proof}

\subsection{Upper bounds}\label{s:upper}
In this subsection we show {Buser-type} upper bounds for $\lmin$ and
$\lmax$ in terms of isoperimetric constants $\hminq{}(d)$ and
$\hmaxq{}(d)$ for arbitrary functions $d$. To this end, for a given
graph $b$, we define for $d:X\times X\to[0,\infty)$
\begin{align*}
  \delta(d):=\inf\{d(x,y)\mid x,y\in X\mbox{ with }b(x,y)>0\}.
\end{align*}

\begin{thm}\label{t:upper1}
For any $p\in [1,\infty)$ and any  function $d:X\times
X\to[0,\infty)$  with $\delta(d)>0$, we have
\begin{align*}
    \lmin\leq \frac{\hminq(d)}{\delta(d)}.
\end{align*}
\end{thm}
\begin{proof}
The inequality directly follows from $b\leq bd/\delta(d)$ and the
equality
\begin{align*}
    \frac{\mathcal{E}_{p}(1_{W})}{\|1_{W}\|_{m,p}^{p}}
    =\frac{|\partial W|_{b}}{m(W)}
\end{align*}
with $1_{W}$ being the characteristic function of a set $W\subseteq
X$.
\end{proof}

\begin{thm}\label{t:upper2}
Let $m(X)<\infty$. For any $p\in [1,\infty)$ and any function
$d:X\times X\to[0,\infty)$ with $\delta(d)>0$ we have
\begin{align*}
    \lmax \leq 2^{p-1}\frac{\hmaxq(d)}{\delta(d)}.
\end{align*},
\end{thm}
\begin{proof} For any set $W\subseteq X$ with $m(W)\leq m(X)/2$ we let
\begin{align*}
    f_{W}=m(X\setminus W)1_{W}-m( W)1_{X\setminus
    W}.
\end{align*}
Then, $f_{W}\in \ell^{\infty}(X)$ and in fact $f\in\ell^{p}(X,m)$
for all $p\in[1,\infty]$ since $m(X)<\infty$. Therefore,
\begin{align*}
\langle{f_{W}},{1}\rangle_m=m(W)m(X\setminus W)-m(W)m(X\setminus
    W)=0
\end{align*}
which implies $\min_{\gm\in \R}\|f-\gm\|_{m,p}=\|f\|_{m,p}$.
Moreover,
\begin{align*}
    \frac{\mathcal{E}_{p}(f_{W})}{{\|f_{W}\|}_{m,p}^{p}}
    =\frac{|\partial W|_{b}(m(X\setminus W)+m(W))^{p}}
    {m(W)m(X\setminus W)^{p}+m(X\setminus W)m(W)^{p}}.
\end{align*}
 By $m(X)<\infty $ we have $\mathcal{E}_{p}(f_{W})<\infty$
and $f\in \ell^{p}(X,m)$ and, therefore, $f\in D_{p}$. Together with
the observations above this yields that
${\mathcal{E}_{p}(f_{W})}/{\|f_{W}\|_{m,p}^{p}}$ is larger than
$\lmax$. Moreover, we apply the inequality $(\al+\beta)^{p}\leq
2^{p-1}(\al^{p}+\beta^{p})$, $\al,\beta\ge0$ and $m(W)\leq
m(X)/2\leq m(X\setminus W)$ to conclude
\begin{align*}
    \lmax \leq 2^{p-1}\frac{|\partial W|_{b}}{m(W)}\leq
 {2^{p-1}}\frac{|\partial W|_{bd}}{\delta(d) m(W)}
\end{align*}
which yields the statement.
\end{proof}

\begin{rem}Let us comment on the constant $\delta(d)$ in the
denominator. The upper bound becomes trivial whenever $\delta
(d)=0$. Suppose that $d\in R_{p}(b,m)$, in which case $\hminq(d)$
and $\hmaxq(d)$ also appears in the lower bounds. In this case
$\delta(d)=0$ whenever $\mathrm{Deg}$ is unbounded, i.e.,
$M=\sup_{x\in X}\mathrm{Deg}(x)=\infty$: Indeed, using the
definition $\mathrm{Deg}(x)=\frac{1}{m(x)}\sum_{y\in X}b(x,y)$, we
see
\begin{align*}
\mathrm{Deg}(x)\leq{\de(d)^{-\frac{p}{p-1}}}\frac{1}{m(x)}\sum_{y\in
X}b(x,y)d(x,y)^{\frac{p}{p-1}}\leq{\de(d)^{-\frac{p}{p-1}}}.
\end{align*}
Thus, $M\leq {\de(d)^{-\frac{p}{p-1}}}$.

 If one considers the isoperimetric constant $h =\inf_{m(W)\leq
m(X)/2}\frac{|\partial W|_{b}}{m(W)}=\hmaxq(1)$ from
Remark~\ref{rem:BH}.(a) instead, then $\delta(1)=1$ and one gets the
upper bound,
\begin{align*}
 \lmax\leq 2^{p-1}\hmaxq(1)
\end{align*}
cf. \cite{BuhHei09}, which does not depend on $M$. However, as
already discussed in Remark~\ref{rem:BH}.(a) this comes   at the
expense of the worse lower bound $(2/M)^{p-1}(\hmaxq(1)/p)^{p}$.
\end{rem}

\subsection{Convergence results for finite
graphs}\label{s:convergence}

In this section we give an alternative proof of the convergence
result $\lmax\to\lmaxq_{1}$, $p\to1$  for finite graphs which is
originally due to B\"uhler/Hein, \cite{HeiBuh10}.

\begin{thm} If $X$ is finite, then  $\lim_{p\downarrow
1}\lm^{(1)}_{p}=\lm^{(1)}_{1}$.
\end{thm}
\begin{proof}
We pick the function $d$ given by
\begin{align*}
    d_{p}(x,y)=(\mathrm{Deg}(x)\vee \mathrm{Deg}(y))^{-(p-1)/p},\quad
    x,y\in X,
\end{align*}
with $\mathrm{Deg}(x)=\sum_{y\in X}b(x,y)/m(x)$. {As discussed in
Example~\ref{example}.(a)}, $d\in R_{p}(b,m)$. By the lower bound in
Theorem~\ref{t:main2} and the upper bound in Theorem~\ref{t:upper2},
we get
\begin{align*}
 \frac{2^{p-1}}{p^{p}}\hmax(d_{p})^{p}& \leq   \lmax\leq
\hmaxq(1)=\lmaxq_{1},
\end{align*}
where a proof of the equality on the right hand side can be carried
over verbatim from \cite[Theorem~2.6]{Chu97} replacing the
normalizing measure by general $m$.
 Since $X$ is finite, there are only finitely many subsets $W$ with
$m(W)\leq m(X)/2$. So, since $d_{p}\to 1$ for $p\to 1$, we deduce
$\hmax(d_{p})\to \hmaxq_{1}(1)$ for $p\to 1$. Thus, it follows
$\lmax\to \lmaxq_{1}$ for $p\to 1$.
\end{proof}

\section{Brook's theorem}\label{s:brooks}
In this section we show an estimate on $\lmin$ from above in terms
of the volume growth of the graph. A result of this type was first
proven by Brooks \cite{Bro81} on manifolds and it was later improved
and generalized in \cite{LiWan01,Stu94}. Similar results were proven
for the normalized Laplacian on graphs in
\cite{DodKar88,Fuj96a,OhnUra94} in the case $p=2$ and in
\cite{Tak03} for general $p$. In \cite{HaeKelWoj13} a corresponding
result for regular Dirichlet forms is proven which unifies all the
above results for $p=2.$ Here, we show a analogous result for
general $p$ and general $p$-Laplacians.

We define  {the exponential volume growth of $X$ by}
\begin{align*}
    \mu=\liminf_{r\to\infty}\inf_{o\in X}\frac{1}{r}\log
   \frac{ m(B_{r}^{(d)}(o))}{m(B_{1}^{(d)}(o))}.
\end{align*}
where $B_{r}^{(d)}(o)$ is the \emph{distance ball} with center $o$
and radius $r$ with respect to a given pseudo metric $d$.

\begin{thm}\label{t:brooks}
Let $p\in[1,\infty)$ and let $d$ be a pseudo metric such that all
distance balls are finite and such that
\begin{align*}
    \sum_{y\in
X}b(x,y)d(x,y)^{p}\leq m(x), \ x\in X.
\end{align*}
Then,
\begin{align*}
    \lm_{p}^{(0)}\leq \frac{\mu^{p}}{2p^{p}}.
\end{align*}
\end{thm}

First, let $d$ be an arbitrary pseudo metric and $\mu$ be the
exponential volume growth defined above. To ease notation we denote
the $r$-balls with center $x_{0}\in X$ by
$B_{r}:=B_{r}^{(d)}(x_{0})$ whenever the pseudo metric $d$ and the
center does not vary.

Next, we construct the test functions following the ideas of
\cite{HaeKelWoj13}. For $r\in\N$, $x_{0}\in X$, $\al>0$, define
\begin{align*}
f_{r,x_{0},\al}&:X\to[0,\infty),\quad x\mapsto  \big( (e^{\al
r}\wedge e^{\al (2r-d(x_{0},x))}) - 1 \big)\vee 0.
\end{align*}
Obviously, we have $f\vert_{B_{r}}\equiv e^{\al r}-1$,
$f\vert_{B_{2r}\setminus B_{r}}= e^{ \al(2r- \rho(x_{0},\cdot))}-1$
and $f\vert_{X\setminus B_{2r}}\equiv0$. Clearly, $f$ is spherically
homogeneous, i.e., there exists $h:[0,\infty)\to[0,\infty)$ such
that $f(x)=h(d(x_{0},x))$.

Moreover, for $r\in\N$, $x_{0}\in X$, $\al>0$, let
$g_{r,x_{0},\al}:X\to[0,\infty)$, be given by
\begin{align*}
g_{r,x_{0},\al}=(f_{r,x_{0},\al}+2)1_{B_{2r}}.
\end{align*}

In \cite{HaeKelWoj13} the following lemma is proven for $p=2$.
However,  with the obvious modifications the proof carries over
verbatim to the case of general $p$.

\begin{lemma}[Lemma~2.2 in \cite{HaeKelWoj13}]\label{l:f} Let $p\in[1,\infty)$
If $\al>\mu/p$, then  there are sequences $(x_{k})$ in $X$ and
$(r_{k})$ in $\N$ such that
$f_{k}=f_{r_{k},x_{k},\al},g_{k}=g_{r_{k},x_{k},\al}\in
\ell^{p}(X,m)$ and we have that $\|g_{k}\|_{m,p}/\|f_{k}\|_{m,p}\to
1$ as $k\to\infty$.
\end{lemma}

The following lemma is also found in \cite{HaeKelWoj13} for $p=2$.
{We sketch a short proof for general $p$}.

\begin{lemma}[Lemma~2.5 in \cite{HaeKelWoj13}]
\label{l:Lip} Let $p\in[1,\infty)$, $r\in\N$, $x_{0}\in X$, $\al>0$
and set $f:=f_{r,x_{0},\al}$, $g:=g_{r,x_{0},\al}$. Then, for
$x,y\in X$
\begin{align*}
(f(x)-f(y))^{p}&\leq \frac{{\al^{p}}}{2}  (g(x)^{p}+g(y)^p)
d(x,y)^{p}.
\end{align*}
\end{lemma}
\begin{proof}
From the second part of the proof of Lemma~\ref{l:f^p}, we see that
$|s^{k}-t^{k}|\leq k|s^{k-1}-t^{k-1}||s-t|/2$ and, therefore,
\begin{align*}
    |e^{s}-e^{t}|&
    \leq {\sum_{k=1}^{\infty}\frac{1}{k!}|s^{k}-t^{k}|}
    \leq
\frac{1}{2}\sum_{k=1}^{\infty}\frac{1}{(k-1)!}(s^{k-1}+t^{k-1})
|s-t|\\
&\leq\frac{(e^{s}+e^{t})}{2}|s-t|
\end{align*}
for $s,t\ge0$. By Jensen's inequality we conclude
\begin{align*}
|e^{s}-e^{t}|^{p}\leq \frac{1}{2}(e^{sp}+e^{tp})|s-t|^{p}
\end{align*}
Without loss of generality we can assume $f(x)\ge f(y)$. Now, we
distinguish the six cases $x,y\in B_{r}$ and $x\in B_{r}, y\in
B_{2r}\setminus B_{r}$ and $x\in B_{r}, y\in X\setminus B_{2r}$ and
$x,y\in B_{2r}$ and $x\in B_{2r}\setminus B_{r}, y\in X\setminus
B_{2r}$ and $x,y\in X\setminus B_{2r}$ to finish the proof.
\end{proof}

\begin{proof}[Proof of Theorem~\ref{t:brooks}]
Let $p\in[1,\infty)$ and $d$ be a pseudo metric such that all
distance balls have finite cardinality and such that
\[
 \sum_{y\in X}b(x,y)d(x,y)^{p}\leq m(x)\quad\hbox{ for all } x\in X\ .
 \]
  Let $\al>\mu/p$ and
let $(f_{k})$, $(g_{k})$ be the sequences of functions given by
Lemma~\ref{l:f}. We see that $f_{k},g_{k}\in C_{c}(X)$, by the assumption of finiteness of the balls. Using the definition of $\lm_{p}^{(0)}$,
Lemma~\ref{l:Lip} and the assumption on $d$, we obtain
\begin{align*}
    \lm_{p}^{(0)}\|f_{k}\|_{m,p}^{p}&\leq \frac{1}{2}\sum_{x,y\in
    X}b(x,y)|f_{k}(x)-f_{k}(y)|^{p}\\
    &\leq \frac{{\al^{p}}}{2}\sum_{x\in X}g_{k}(x)^{p}\sum_{y\in
    X}b(x,y)d(x,y)^{p}\\
    &\leq \frac{\al^{p}}{2}\|g_{k}\|_{m,p}^{p}.
\end{align*}
Since by~Lemma~\ref{l:f} $\|f_{k}\|_{m,p}/\|g_{k}\|_{m,p}\to 1$ as
$k\to\infty$ we now deduce $\lm_{p}^{(0)}\leq \al^{p}/2$ for all
$\al>\mu/p$. Thus, the statement of the theorem follows.
\end{proof}

\section{Appendix}

The \emph{ discrete $p$-Laplacian} $\mathcal{L}_{p}$ is a quasi-linear (linear for $p=2$) operator defined by
\begin{align*}
    \mathcal{L}_{p}f(x)=\frac{1}{m(x)}\sum_{y\in X}b(x,y)|f(x)-f(y)|^{p-2}(f(x)-f(y))\ .
\end{align*}
on
\[
 \mathcal{F}_{p}:=\{f\in C(X)\mid \sum_{y\in X}b(x,y)|f(y)|^{p-1}<\infty \mbox{ for all }x\in X\}.
\]
The parabolic equation
\begin{equation}\label{i:eq-parabolic}
\frac{d }{dt}u(t,x)+\mathcal L_p u(t,x)=0,\qquad t> 0,\ x\in X\ ,
\end{equation}
associated with this operator has received some attention lately, cf.~\cite{Mug13,HuaMug15} and the references therein. One can naturally regard this as an evolution equation in the Hilbert space $\ell^2(X;m)$: mimicking the techniques of~\cite[\S~4.4]{DibUrbVes04} one sees that its long-time behavior is determined by the Rayleigh-type quotient
\[
\frac{\Efun_{p}(\ph)}{\|\ph\|^{2}_{m,2}}
\]
evaluated at $\ph=u_0$, the initial data of the above problem. Observe that this functional is not homogeneous, and in fact this is \emph{not} the quotient we have considered throughout this paper. Indeed, homogeneity is an important property of energy functionals and its lack significantly complicates the parabolic theory of~\eqref{i:eq-parabolic}, which suggests to introduce the relevant functional of this paper,
\begin{equation}\label{eq:trud}
\frac{\Efun_{p}(\ph)}{\|\ph\|^{p}_{m,p}}\ .
\end{equation}

In search of homogenization procedures that would allow for a non-linear extension of linear Harnack inequalities, it was observed in~\cite{Tru68} that~\eqref{i:eq-parabolic} should be replaced by what is now occasionally called the \emph{Trudinger equation}: in our setting it reads
\begin{equation}\label{i:eq-parabolic-2}
\frac{d }{dt}(u^{p-1})(t,x)+\mathcal L_p u(t,x)=0,\qquad t> 0,\ x\in X\ ,
\end{equation}
whose corresponding eigenvalue equation~\eqref{eq:weaksol} is associated with the functional in~\eqref{eq:trud}.
In analogy with the classical theory of Laplacians on domains of $\mathbb R^d$, we actually can think of the quantities $\lmin$ and $\lmax$ introduced in this paper as the spectral gap of the Dirichlet and Neumann realizations of the $p$-Laplacian $\mathcal L_p$.

\bibliographystyle{alpha}
\bibliography{../../referenzen/literatur}

\end{document}